\documentclass[11pt]{article}
\usepackage[utf8]{inputenc}
\usepackage{amsmath, amsthm, amssymb}
\usepackage{geometry}
\usepackage{hyperref}
\usepackage{graphicx}
\usepackage{cite}

\title{The dynamical Alekseevskii conjecture in dimension five}
\author{Maher Billon}
\date{\today}

\newtheorem{theorem}{Theorem}[section]

\newtheorem{lemma}[theorem]{Lemma}
\newtheorem{corollary}[theorem]{Corollary}

\newtheorem{prop}[theorem]{Proposition}
\newtheorem{claim}[theorem]{Claim}
\newtheorem{conjecture}[theorem]{Conjecture}

\begin{document}
\maketitle
\begin{abstract}
    We prove the dynamical Alekseevski conjecture in dimension five. We also provide a detailed analysis of the homogeneous Ricci flows on $SO(3)\ltimes \mathbb{R}^3/SO(2)$ and $SL(2,\mathbb{C})/U(1)$.
\end{abstract}
\section{Introduction}
The Ricci flow on a smooth Manifold is defined as the maximal solution to the partial differential equation \[
\frac{dg}{dt}=-2\mathrm{Ric}(g),
\]
where $g(t)$ denotes a time-dependent Riemannian metric and $\mathrm{Ric}(g)$ is its Ricci curvature. A solution is said to be immortal if it exists for all positive times $t$.
\par
On homogeneous spaces, the Ricci flow evolves within the finite-dimensional space of homogeneous metrics. This restriction simplifies the computations and has made the homogeneous setting a fruitful ground for investigating long term behavior of the flow. In this context, an alternative formulation known as the bracket flow [Lauret, \cite{Lauret2013}] has proven particularly effective. Using this framework, it was shown by Lafuente \cite{Lafuente2015} that if the universal cover of a homogeneous space is diffeomorphic to $\mathbb{R}^n$, then the Ricci flow is immortal.
\par The converse implication was conjectured by Böhm and Lafuente : 
\begin{conjecture}[Dynamical Alekseevski conjecture,  \cite{Boehm2018}] If a homogeneous space has an immortal Ricci flow, then its universal cover is diffeomorphic to $\mathbb{R}^n$.
\end{conjecture}
It is in fact the "dynamical" version of the well known Alekseevski conjecture stating that a simply-connected homogeneous negative-Einstein space is diffeomorphic to $\mathbb{R}^n$. In that version, we replace a static condition $Ric(g)=\lambda g$ with a dynamic condition $\frac{dg}{dt}=-2\mathrm{Ric} (g)$. The two conjectures are clearly equivalent for Einstein spaces, on which the Ricci flow is given by $g(t)=(1-2\lambda t)g_0$.
\par An immediate corollary is an "all or nothing" law for homogeneous flows: either the space has a universal cover diffeomorphic to $\mathbb{R}^n$ and all the homogeneous flows are immortal, or all the homogeneous flows have a finite time extinction.
\par
Up to this date, the dynamical Alekseevski conjecture has been verified in several settings, including symmetric spaces, compact homogeneous spaces [Böhm, \cite{Boehm2015}], and cases where the isometry group $G$ has a compact normal semisimple subgroup [Araujo, \cite{Araujo2024}]. However, a full understanding of the long-time behaviour of homogeneous Ricci flows remains incomplete. In dimensions $\leq 4$, the conjecture has been fully addressed \cite{Araujo2024}. Dimension $5$ is the first setting where simply connected, non Euclidean homogeneous spaces exist that are not covered by previous results.
\par The main result of this paper is the following theorem:
\begin{theorem} In dimension $5$, the dynamical Alekseevski conjecture holds.
\end{theorem}
\par Our strategy is to classify systematically all five-dimensional homogeneous spaces and to identify
the exceptional cases that fall outside known criteria (\textbf{Part 2}). We will isolate two such exceptional spaces: $SO(3)\ltimes \mathbb{R}^3/SO(2)$ and $SL(2,\mathbb{C})/U(1)$. For both cases, we will derive explicit Ricci flow equations and show that the flow develops a singularity in finite time, i.e, the flow is not immortal (\textbf{Parts 3 and 4}).
\par As a consequence, we will easily obtain a result on the topology of the space of positively curved metrics on $5-$dimensional homogeneous spaces (\textbf{Part 5}):
\begin{corollary}
    For a five-dimensional homogeneous space $M=G/H$, the set of homogeneous metrics with a positive scalar curvature $M^G_{\mathrm{Scal} > 0}$ is either empty or contractible.
\end{corollary}
\textit{Acknowledgments. }I'm grateful to my supervisor, Ramiro Lafuente, for his continuous support and helpful suggestions.
 \section{Classification of 5-dimensional homogeneous spaces}
The Ricci flow is immortal on $M$ if and only if it is immortal on $\tilde{M}$. So we will directly suppose that $M$ is simply connected. We write $M=G/H$ and we will assume some properties on this writing (that we can always assume without loss of generality):
\begin{itemize}
    \item[i)] $G$ doesn't contain a proper Lie subgroup acting transitively by isometries on $M$;
    \item[ii)] One Levi decomposition $G=G_{ss} \ltimes R$ is such that $H \subset G_{ss}$;
    \item[iii)] $G_{ss}$ acts effectively on $G_{ss}/H$;
    \item[iv)] $H$ is compact.
\end{itemize}
On the Lie algebra level, we have $\mathfrak{g} = \mathfrak{g}_{ss}\ltimes \mathfrak{r}$ and $\mathfrak{h} \subset \mathfrak{g}_{ss}$ with $codim_{\mathfrak{g}_{ss}}(\mathfrak{h}) + dim(\mathfrak{r}) = 5$. Notice that \[
M \cong G_{ss}/H \times R
\]
and simply connected solvable Lie groups are Euclidean. Therefore, if $G_{ss}/H$ is diffeomorphic to $\mathbb{R}^k$ then $M$ is diffeomorphic to $\mathbb{R}^5$. 
\begin{prop}
    Under the above assumptions, one of the following holds :
    \begin{itemize}
        \item $M$ is diffeomorphic to $\mathbb{R}^5$;
        \item $G$ has a semisimple normal compact subgroup;
        \item $M$ is symmetric;
        \item $M$ is a Riemannian product of two homogeneous spaces;
        \item $M = SL(2,\mathbb{C})/U(1)$ or $M=SO(3)\ltimes \mathbb{R}^3/SO(2)$ with a homogeneous metric.
    \end{itemize}
\end{prop}
We will go through the cases according to the dimension of $r$.
\subsection{Dim(r)=5}
In that case, $H=G_{ss}$ so by \textbf{iii}, $dim(G_{ss})=0$ and $M=G=R$ so $M\cong \mathbb{R}^5$.
\subsection{Dim(r)=4}
In that case, $codim_{\mathfrak{g}_{ss}}(\mathfrak{h}) = 1$. It is impossible as shown by the following lemma :
\begin{lemma}
    If $\mathfrak{g}_{ss}$ is a semisimple Lie algebra, then it cannot contain a codimension$-1$ compact embedded subalgebra.
\end{lemma}
\begin{proof}
    If $k$ is a $1$-codimensional compact subalgebra, then it has an $ad(k)-$ invariant complementary space $p=Vect(x)$. But then $p$ is a solvable ideal, contradiction.
\end{proof}
\subsection{Dim(r)=3}
In that case, $codim_{G_{ss}}(H) = 2$ so $G_{ss}/H$ is a two-dimensional Riemannian homogeneous space i.e it is homothetic to $\mathbb{S}^2$, $\mathbb{R}^2$ or $\mathbb{H}^2$ with their canonical metrics.
\begin{itemize}
    \item If it's $\mathbb{R}^2$ or $\mathbb{H}^2$ then $M\cong \mathbb{R}^5$.
    \item If it's $\mathbb{S}^2$, then by \textbf{iii} $\mathfrak{g}_{ss} \hookrightarrow Lie(Isom(\mathbb{S}^2)) = \mathfrak{so}(3)$ is injective. Therefore $g_{ss}$ is a semisimple subalgebra of $\mathfrak{so}(3)$, so it has to be $\mathfrak{so}(3)$ because of the dimensions. In that case, $\mathfrak{h}=\mathfrak{so}(2)$ is the whole isotropy group. So $\mathfrak{g} = \mathfrak{so}(3) \ltimes \mathfrak{r}$ with a specific semidirect product, $\mathfrak{h}=\mathfrak{so}(2)\subset \mathfrak{so}(3)$ and $M = SO(3)\ltimes R/SO(2)$.
\end{itemize}
    \begin{lemma}
        Either $\mathfrak{so}(3) \ltimes \mathfrak{r}$ is a direct product, or $\mathfrak{r}=\mathbb{R}^3$ and the semidirect product is the standard semidirect product $\mathfrak{so}(3) \ltimes \mathbb{R}^3$.
    \end{lemma}
    \begin{proof}
        The representation of $\mathfrak{so}(3)$ involved in the semidirect product is either trivial or the standard $3$-dimensional irreducible representation of $\mathfrak{so}(3)$. In the first case, the semidirect product is a direct product. In the second case, we have a vector basis $(e_1, e_2, e_3)$ of  $r$ such that $\mathfrak{so}(3)$ acts like 
        \[
E = \begin{pmatrix}
0 & 1 & 0 \\
-1 & 0 & 0 \\
0 & 0 & 0
\end{pmatrix}
\quad
F = \begin{pmatrix}
0 & 0 & 1 \\
0 & 0 & 0 \\
-1 & 0 & 0
\end{pmatrix}
\quad
G = \begin{pmatrix}
0 & 0 & 0 \\
0 & 0 & 1 \\
0 & -1 & 0
\end{pmatrix}
\]
Then we use that $\mathfrak{so}(3)$ must act as derivations of $R$ :
$$E[c_1,c_2] = [Ec_1, c_2] + [c_1, Ec_2] = 0$$
$$F[c_1,c_3] = [Fc_1, c_3] + [c_1, Fc_3] = 0$$
$$G[c_2,c_3] = [Gc_2, c_3] + [c_2, Gc_3] = 0$$
So $[c_1,c_2] \in Vect(c_3)$, $[c_2,c_3] \in Vect(c_1)$, $[c_3,c_1] \in Vect(c_2)$. But $\mathfrak{r}$ is solvable, so one of the three brackets must vanish. Without loss of generality, let's say $[c_2,c_3] = 0$. Then :
$$E[c_1,c_3] = -[c_2,c_3] + [c_1, Ec_3]= 0$$
so $[c_1,c_3]=0$ and same $[c_1,c_2]=0$. Finally, $r = \mathbb{R}^3$.
    \end{proof}
    In the first case, $G$ contains a normal compact semisimple subgroup. In the second case, $M=SO(3)\ltimes \mathbb{R}^3/SO(2)$ with the standard semidirect product.
\subsection{Dim(r)=2}
In that case, $codim_{\mathfrak{g}_{ss}}(\mathfrak{h})=3$. By \textbf{iii}, $\mathfrak{g}_{ss}\hookrightarrow Lie(Isom(G_{ss}/H))$ is injective. We also know that $dim(Isom(G_{ss}/H)) \leq \frac{3\cdot 4}{2} = 6$. Moreover, either $G_{ss}/H$ has a constant sectional curvature (and therefore is isometric to $\mathbb{R}^3, \mathbb{H}^3$ or $\mathbb{S}^3$ up to homothety) or $dim(Isom(G_{ss}/H)) \leq \frac{3\cdot 2}{2}+1 = 4$. By looking at the semisimple Lie algebras of dimension less than $6$  we have the following cases :
\begin{itemize}
    \item If $\mathfrak{g}_{ss} = \mathfrak{so}(3)$ then $\mathfrak{h}$ is trivial and $M = SU(2) \ltimes R$ with a homogeneous metric. The only representation of $\mathfrak{so}(3)$ of dimension less than $3$ is the trivial representation, so $M=SU(2)\times R$ and $G$ has a normal semisimple compact subgroup.
    \item If $\mathfrak{g}_{ss} = \mathfrak{sl}(2,\mathbb{R})$ then $\mathfrak{h}$ is trivial and $M=\tilde{SL(2,R)} \ltimes R$ so $M\cong \mathbb{R}^5$.
    \item If not, $dim(\mathfrak{g}_{ss}) = 6$ so we are in the equality case and $G_{ss}/H$ is isometric to $\mathbb{R}^3, \mathbb{H}^3$ or $\mathbb{S}^3$ with their natural metrics (up to homothety). In the two first cases, $M\cong \mathbb{R}^5$. In the last case, $\mathfrak{g}_{ss}\hookrightarrow Lie(Isom(G_{ss}/H))$ is one-to-one so $g_{ss} = \mathfrak{so}(4)=\mathfrak{so}(3)+\mathfrak{so}(3)$. Same as before, the semidirect product is a direct product. But then $G$ has a normal compact semisimple subgroup $SO(4)$.
\end{itemize}
\subsection{Dim(r)=1}
In that case, $codim_{\mathfrak{g}_{ss}} (\mathfrak{h}) = 4$ and $\mathfrak{r} = \mathbb{R}$. The only $1$-dimensional representation of a semisimple Lie algebra is the trivial one (because the kernel of the representations contains $[g_{ss},g_{ss}]=g_{ss}$), so $G = G_{ss}\times R$.
Just like before, $g_{ss}\hookrightarrow Lie(Isom(G_{ss}/H))$ is injective and $dim(Isom(G_{ss}/H)) \leq \frac{4\cdot 5}{2} = 10$. Moreover, either $G_{ss}/H$ has a constant sectional curvature (and therefore is isometric to $\mathbb{R}^4, \mathbb{H}^4$ or $\mathbb{S}^4$) or $dim(Isom(G_{ss}/H)) \leq 9$.
\begin{itemize}
    \item If $dim(G_{ss})=10$ then $G_{ss}/H$ is either $\mathbb{R}^4, \mathbb{H}^4$ or $\mathbb{S}^4$. In the two first cases, $M$ is diffeomorphic to $\mathbb{R}^5$. In the last case, $G_{ss}=SO(5)$ so $G$ has a compact normal subgroup .
    \item If $dim(G_{ss}) = 9$ then $\mathfrak{g}_{ss}$ is a sum of $\mathfrak{so}(3)$, $\mathfrak{sl}(2,R)$, $\mathfrak{sl}(2,C)$ and $dim(\mathfrak{h})=5$. But the dimension of a maximal compact subgroup of $so(3)$ or $sl(2,R)$ is $1$, and the dimension of a maximal compact subgroup of $\mathfrak{sl}(2,\mathbb{C})$ is $3$. So $\mathfrak{h}$ cannot be $5-$dimensional.
    \item If $dim(G_{ss}) = 8$ then $\mathfrak{g}_{ss}=\mathfrak{su}(2,1)$ or $\mathfrak{sl}(3,R)$ and $dim(\mathfrak{h})=4$. But the dimension of a maximal compact subgroup of $\mathfrak{sl}(3,R)$ is $3$, so this case is not possible. Moreover, the maximal compact subgroup of $SU(2,1)$ is $S(U(2)\times U(1))$ and has dimension $4$. So in that case $M$ is Euclidean.
    \item The remaining possibilities are $\mathfrak{g}_{ss} = \mathfrak{so}(3)+\mathfrak{so}(3)$, $\mathfrak{so}(3)+\mathfrak{sl}(2,R)$, $\mathfrak{sl}(2,R)+\mathfrak{sl}(2,R)$, $\mathfrak{sl}(2,C)$. If $\mathfrak{so}(3)$ is involved we are in the case where $G_{ss}$ has a compact normal semisimple subgroup. If $g_{ss} = sl(2,R)+sl(2,R)$ then $H$ is a maximal compact subgroup of $G_{ss}$ so $M$ is Euclidean. If $g_{ss} = sl(2,C)$ then we can suppose $\mathfrak{h}\subset \mathfrak{su}(2)$ (maximal compact subgroup of $SL(2,C)$). But $\mathfrak{su}(2)$ is simple so it cannot have a codimension$-1$ compact subalgebra (\textbf{lemme} 
 \textbf{2.1}). Contradiction.
\end{itemize}
\subsection{Dim(r)=0}
In that case, $G=G_{ss}$ and $M=G_{ss}/H$. 
\begin{itemize}
    \item If $G_{ss}$ has a compact factor in its decomposition in simple groups, then $G$ has a compact normal semisimple subgroup.
    \item If $M$ is a Riemannian product of two homogeneous spaces, then the proposition is true.
    \item If $M$ is symmetric, then the proposition is true.
    \item From the work of Arroyo and Lafuente [\cite{Arroyo2017}, Table 1], we get the remaining cases $M=SL(2,C)/U(1)$, $M=SL(2,R)\times SL(2,R)/\Delta_{p,q}SO(2)$ and $M=SU(2,1)/SU(2)$. For the last one, $SU(2)$ has codimension $1$ in a maximal compact subgroup of $SU(2,1)$ so $M\cong \mathbb{R}^5$. For $SL(2,R)\times SL(2,R)/\Delta_{p,q}SO(2)$, the space is clearly diffeomorphic to $\mathbb{R}^5$ because the maximal compact subgroup is $SO(2)\times SO(2)$ and $SO(2)\times SO(2)/\Delta_{p,q}SO(2)$ has a universal cover diffeomorphic to $\mathbb{R}$. The last interesting case is $M=SL(2,C)/U(1)$ with a homogeneous metric.
\end{itemize}
\section{First special case : $SO(3)\ltimes R^3/SO(2)$}
\subsection{Lie algebra structure}
We use the unique $3$-dimensional representation of $\mathfrak{ so}(3)$, with basis :
\[
E = \begin{pmatrix}
0 & 1 & 0 \\
-1 & 0 & 0 \\
0 & 0 & 0
\end{pmatrix}
\quad
F = \begin{pmatrix}
0 & 0 & 1 \\
0 & 0 & 0 \\
-1 & 0 & 0
\end{pmatrix}
\quad
G = \begin{pmatrix}
0 & 0 & 0 \\
0 & 0 & 1 \\
0 & -1 & 0
\end{pmatrix}
\]
We note $c_1, c_2, c_3$ the canonical basis of $\mathbb{R}^3$.
Then the Lie algebra $g$ is generated by $(E,F,G,c_1, c_2, c_3)$ with bracket specified in the following table.
\begin{center}
\begin{tabular}{c|cccccc}
$[\quad]$ & $E$ & $c_3$ & $c_1$ & $c_2$ & $F$ & $G$ \\
\hline
$E$ & $0$ & $0$ & $-c_2$ & $c_1$ & $-G$ & $F$ \\
$c_3$ & $0$ & $0$ & $0$ & $0$ & $-c_1$ & $-c_2$\\
$c_1$ & $c_2$ & $0$ & $0$ & $0$ & $c_3$ & $0$\\
$c_2$ & $-c_1$ & $0$ & $0$ & $0$ & $0$ & $c_3$ \\
$F$ & $G$ & $c_1$ & $-c_3$ & $0$ & $0$ & $-E$ \\
$G$ & $-F$ & $c_2$ & $0$ & $-c_3$ & $E$ & $0$

\end{tabular}
\end{center}
In this setting we can assume that the isotropy subalgebra is $\mathfrak{so}(2) = Vect(E)$. Moreover, $p = Vect(F,G,c_1, c_2, c_3)$ is an $ad(E)$-invariant complement of $Vect(E)$ in $g$.\\
It is easy to show that the decomposition of $p$ in $\mathfrak {so}(2)$ irreducible representations is $p = Vect(c_3) + Vect(c_1,c_2) + Vect(F,G)$. The problem making this case hard to solve is that the two $2$-dimensional representations are isomorphic, so cross-terms will appear in the metric and the Ricci tensor.\\
It will be useful to have the killing form on $p$ :
\begin{center}
\begin{tabular}{c|ccccc}
B & $c_3$ & $c_1$ & $c_2$ & $F$ & $G$ \\
\hline
$c_3$ & $0$ & $0$ & $0$ & $0$ & $0$\\
$c_1$ & $0$ & $0$ & $0$ & $0$ & $0$\\
$c_2$ & $0$ & $0$ & $0$ & $0$ & $0$ \\
$F$ & $0$ & $0$ & $0$ & $-4$ & $0$ \\
$G$ & $0$ & $0$ & $0$ & $0$ & $-4$

\end{tabular}
\end{center}
Finally, we see that $tr \circ ad = 0$ so $G$ is unimodular.
\subsection{Homogeneous metrics}
We want to characterize all the the homogeneous metrics on $SO(3)\ltimes \mathbb{R}^3/SO(2)$. Such a metric correspond to an $\mathrm{Ad}(SO(2))-$invariant scalar product on $p$. Using this invariance and the irreducible decomposition of $p$, we get that the scalar product in the basis $(c_3, c_1, c_2, F,G)$ has the following form :
\[
g = \begin{pmatrix}
    \alpha & 0 & 0 & 0 & 0 \\
    0 & \beta & 0 & \mu & \nu \\
    0 & 0 & \beta & -\nu & \mu \\
    0 & \mu & -\nu & \gamma & 0 \\
    0 & \nu & \mu & 0 & \gamma \\
\end{pmatrix},
\]
with $\alpha, \beta, \gamma > 0 $.
As a matter of fact, we can get rid of the parameter $\mu$.
\begin{lemma}
    Every homogeneous metric on $SO(3)\ltimes \mathbb{R}^3/SO(2)$ is isometric to a homogeneous metric corresponding to a scalar product on $p$ where $\mu=0$. 
\end{lemma}
\begin{proof}
    Let $x = \mathrm{exp}(tc_3)$ and $\alpha_t$ the conjugation by $x$ in our Lie group. We know that $(d\alpha_t)_e = Ad(tc_3) = \mathrm{exp}(ad(tc_3))$ is the identity on $Vect(E)$ so $\alpha_t (\mathrm{exp}(\lambda E)) = \mathrm{exp}(\lambda E) \forall \lambda$ so $\alpha_t$ is the identity on $SO(2)^0 = SO(2)$. Therefore, $\alpha_t$ induces $\tilde{\alpha_t}$ on the quotient. Clearly, $\tilde{\alpha_t}$ is a diffeomorphism. Besides, $\tilde{\alpha_t}^*g $ is a $G$-invariant metric on $G/H$. Finally, the corresponding scalar product on $p$ is given by $\tilde{g} = g(Ad(tc_3)\cdot, Ad(tc_3)\cdot)$.\\
    We compute in the same basis :
    \[
    ad(tc_3) = \begin{pmatrix}
        0 & 0 & 0 & 0 & 0 \\
        0 & 0 & 0 & -t & 0 \\
        0 & 0 & 0 & 0 & -t \\
        0 & 0 & 0 & 0 & 0 \\
        0 & 0 & 0 & 0 & 0 \\
    \end{pmatrix}
    \quad 
    \mathrm{Ad}(\mathrm{exp}(tc_3)) = \begin{pmatrix}
        1 & 0 & 0 & 0 & 0 \\
        0 & 1 & 0 & -t & 0 \\
        0 & 0 & 1 & 0 & -t \\
        0 & 0 & 0 & 1 & 0 \\
        0 & 0 & 0 & 0 & 1 \\
    \end{pmatrix}
    \]
    \[
    \tilde{g} = \begin{pmatrix}
        \alpha & 0 & 0 & 0 & 0 \\
        0 & \beta & 0 & \mu-t\beta & \nu \\
        0 & 0 & \beta & -\nu & \mu-t\beta \\
        0 & \mu-t\beta & -\nu & \gamma + t^2\beta - 2t\mu & 0 \\
        0 & \nu & \mu-t\beta & 0 & \gamma + t^2\beta - 2t\mu \\
    \end{pmatrix}
    \]
    Then choose $t=\frac{\mu}{\beta}$ to finish the proof.
\end{proof}
We now have to find an orthonormal basis for $g$. We introduce $\tau = \mu^2 + \nu^2$. Let's suppose that $\tau > 0$. We define 
\[
\lambda_{\pm} = \frac{\beta+\gamma}{2} \pm \frac{1}{2} \sqrt{4\tau + (\beta - \gamma)^2}
\]
These are the roots of the equation $(\lambda - \beta)(\lambda - \gamma) = \tau$ and are distinct from $\beta, \gamma$ by assumption.\\
Let's introduce \[
H^{\pm} = \begin{pmatrix}
0\\
    \lambda_{\pm} - \gamma \\
    0\\
    \mu \\
    \nu
\end{pmatrix}
\quad 
B^{\pm} = \begin{pmatrix}
0\\
    0 \\
    \lambda_{\pm} - \gamma\\
    -\nu \\
    \mu
\end{pmatrix}
\]
Then $(H^{\pm}, B^{\pm})$ are eigenvectors of $g$ associated to eigenvalues $\lambda_{\pm}$ (simple computation). It is then straightforward to show that $(\frac{c_3}{\sqrt{\alpha}},\frac{H^{\pm}}{f_{\pm}}, \frac{B^{\pm}}{f_{\pm}})$ is an orthonormal basis for $g$, where \[
f_{\pm} = \sqrt{\lambda_{\pm} \cdot (\tau + (\lambda_{\pm} - \gamma)^2)}
\]
\subsection{Ricci tensor}
We know from [\cite{Besse1987}, Corollary 7.38] that if $(X_i)$ is an orthonormal basis of $p$ for $g$ then the Ricci tensor in the unimodular case is given by :
\[
Ric(X,Y) = -\frac{1}{2}B(X,Y) - \frac{1}{2}\sum_i g([X,X_i],[Y,X_i]) + \frac{1}{4}\sum_{i,j}g([X_i, X_j],X)g([X_i,X_j],Y)
\]
With that formula and simplifications with the $\lambda_{\pm}$, we are able to express the full Ricci tensor (which has the same form as the metric $g$):
\[
Ric(c_3, c_3) = \frac{\alpha^2 - \beta^2}{\beta\gamma - \tau} + \frac{2\alpha^2 \nu^2}{(\beta\gamma - \tau)^2}
\]
\[
Ric(c_1, c_1) = Ric(c_2, c_2) = \frac{\beta}{2\alpha}\frac{\beta^2 - \alpha^2}{\beta\gamma - \tau}
\]
\[
Ric(F,F) = Ric(G,G) = 2 - \frac{\beta}{\alpha} + \frac{\gamma}{2\alpha}\frac{\beta^2 - \alpha^2}{\beta\gamma - \tau}
\]
\[
Ric(c_1, F) = Ric(c_2, G) = \frac{\mu}{2\alpha (\beta\gamma - \tau)} (\beta^2 - \alpha^2)
\]
\[
Ric(c_1, G) = -Ric(c_2, F) = \frac{\nu }{2\alpha (\beta\gamma - \tau)} (\alpha^2 + \beta^2)
\]
\subsection{Extinction time of the Ricci flow}
The Ricci flow equation is $\frac{dg}{dt} = -2Ric(g)$ i.e
\[
\begin{cases}
\displaystyle \dot{\alpha} = 2\frac{\beta^2 - \alpha^2}{\beta\gamma - \tau} - \frac{4\alpha^2 \nu^2}{(\beta\gamma - \tau)^2} \\
\displaystyle \dot{\beta} = \frac{\beta}{\alpha}\frac{\alpha^2 - \beta^2}{\beta\gamma - \tau} \\
\displaystyle \dot{\gamma} = -4 + 2\frac{\beta}{\alpha} + \frac{\gamma}{\alpha} \frac{\alpha^2 - \beta^2}{\beta\gamma - \tau} \\
\displaystyle \dot{\mu} = \frac{\mu}{\alpha (\beta\gamma - \tau)}(\alpha^2 - \beta^2) \\
\displaystyle \dot{\nu} = -\frac{\nu}{\alpha(\beta\gamma - \tau)}(\alpha^2 + \beta^2)
\end{cases}
\]
We will now prove the main theorem of this section.
\begin{theorem}
    Any homogeneous Ricci flow on $SO(3)\ltimes \mathbb{R}^3/SO(2)$ has a finite time extinction.
\end{theorem}
First thing to notice is that if $\mu = 0$ at $t=0$ then $\mu=0$ at any time due to the differential equation. Combined with the observations in Section $2$, we can therefore assume $\mu (t)=0$ and $\tau = \nu^2$.\\
Let's suppose that the Ricci flow above is immortal. We should arrive to some contradiction. There are two relevant ratios to analyze, which are $x = \frac{\beta}{\alpha}$ and $\epsilon = \frac{\tau}{\beta \gamma}$. The constraints on these variables are $x>0$ and $\epsilon \in [0,1)$.
\begin{claim}
    If $\tau = 0$ then the flow has a finite time extinction.
\end{claim}
\begin{proof}
In that case, $\mu = \nu = 0$.
Clearly, if $\beta = \alpha$ at a certain time then they remain equal forever. If this never happens, then $\alpha$ and $\beta$ are monotonic because their derivatives have the sign of $\pm (\beta-\alpha)$ and are null if and only if $\alpha=\beta$. Therefore $\alpha, \beta$ converge to constants $\alpha_{\infty}, \beta_{\infty}>0$. Note that $\dot{\gamma} = -4 + \frac{\alpha}{\beta} + \frac{\beta}{\alpha}$ and  $\frac{d}{dt} (\frac{\gamma}{\beta}) = \frac{2}{\beta} (\frac{\beta}{\alpha} - 2)$. So if $\frac{\beta_{\infty}}{\alpha_{\infty}} < 2$ then $\frac{\gamma}{\beta} \rightarrow -\infty$. If $2\leq \frac{\beta_{\infty}}{\alpha_{\infty}} < 2+\sqrt{3}$ then $\gamma \rightarrow {-\infty}$. If $\frac{\beta_{\infty}}{\alpha_{\infty}} \geq 2+\sqrt{3}$ then $\frac{\gamma}{\beta} \sim \frac{2}{\beta_{\infty}}(\frac{\beta_{\infty}}{\alpha_{\infty}}- 2)t$ so $\dot{\beta} \sim -\frac{\alpha_{\infty} (\frac{\beta_{\infty}}{\alpha_{\infty}}^2 - 1)}{2(\frac{\beta_{\infty}}{\alpha_{\infty}} - 2)t}$ so $\beta \rightarrow -\infty$. In each case we come to a contradiction.
\end{proof}
Now we can assume $\tau(t)>0$.
\begin{claim}
        $\epsilon$ is strictly decreasing.
\end{claim}
\begin{proof}
    We compute \[
    \frac{\dot{\epsilon}}{\epsilon} = -\frac{2\alpha}{ (1-\epsilon) \beta \gamma } ((1-\epsilon)x^2 - (2\epsilon - 2)x + 2)
    \]
    The discriminant of this polynomial is $\Delta = -8\epsilon (1-\epsilon)<0$ so $\frac{\dot{\epsilon}}{\epsilon} < 0$.
\end{proof}
\begin{claim}
    $\exists t_0$ such that either $\forall t>t_0,x(t)>1$ or $\forall t>t_0, x(t) < 1$.
\end{claim}
\begin{proof}
    If $\alpha = \beta$ then $\dot{\beta} = 0$ and $\dot{\alpha} < 0$ so just before $x<1$ and just after $x>1$. Therefore it can only happen $1$ time and after that the sign of $x-1$ is preserved.
\end{proof} 
\begin{claim}
    $x$ converges to a constant $x_{\infty}\in ]0,+\infty[$.
\end{claim}
\begin{proof}
    First thing to do is computing the derivative of $x$:
    \[
    \dot{x} = \frac{3\beta}{\beta \gamma - \tau} (\frac{4\epsilon}{3(1-\epsilon)} - (x^2-1))
    \]
    First thing to note is that $\frac{4\epsilon}{3(1-\epsilon)}$ is a strictly increasing function of $\epsilon$.\\
    If $t_0$ is a critical time for $x$ (i.e $\dot{x}(t)=0$) then by developping $\dot{x}$ around $t_0$ at the first order, we see that $\dot{x}>0$ for $t_0-\delta < t<t_0$ and $\dot{x}<0$ for $t_0 < t<t_0+\delta$. Therefore, $x(t_0)$ is a unique global maximum. So after a certain time, $x$ is monotonic.\\
     But $\frac{4\epsilon}{3(1-\epsilon)}$ is bounded (because $\epsilon$ decreases) so if $x$ is too large, $x$ decreases. Moreover, if $x\leq 1$ then clearly $x$ increases. Therefore, $x$ cannot neither diverge nor converge to $0$ and $x\to x_{\infty}\in ]0,+\infty[$.
\end{proof}
 A short computation shows that \[
\dot{(\frac{\gamma}{\beta})} = \frac{2}{\beta} (\frac{\beta}{\alpha} - 2)
\]
If $x_{\infty} \leq 1$ and $x<1$ after a certain time, then $\alpha$ decreases, $\beta$ increases and $\beta < \alpha$ so $\alpha, \beta$ converge to strictly positive constants. Therefore, $\dot{(\frac{\gamma}{\beta})} \rightarrow \frac{2}{\beta_{\infty}}(x_{\infty} - 2 <0)$ so $\frac{\gamma}{\beta}\rightarrow -\infty$ which is a contradiction. \\
If $x_{\infty}\leq 1$ and $x>1$ after a certain time then $x_{\infty}=1$ and \[
\dot{\gamma} \leq -4 + 2\frac{\beta}{\alpha} + \frac{\gamma}{\alpha} \frac{\alpha^2 - \beta^2} {\beta \gamma} = -4 + x + \frac{1}{x}
\] so after a while $\dot{\gamma} < -1$ and $\gamma \rightarrow -\infty$, contradiction.  \\
If $1<x_{\infty}<2+\sqrt{3}$ then after a certain time, $\beta > \alpha$ and \[
\dot{\gamma} \leq -4 + 2\frac{\beta}{\alpha} + \frac{\gamma}{\alpha} \frac{\alpha^2 - \beta^2} {\beta \gamma} = -4 + x + \frac{1}{x}
\]
The last term is null for $x = 2+\sqrt{3}$ and strictly negative for $1<x<2+\sqrt{3}$ so after a certain time, $\dot{\gamma} \leq -cste < 0$, and we would have $\gamma \rightarrow -\infty$, which is a contradiction.\\
The last case is $x_{\infty} \geq 2 + \sqrt{3}$. In that case, after a while, $\beta > \alpha$ so $\beta$ is decreasing.
Therefore $\dot{(\frac{\gamma}{\beta})} \geq cste > 0$ and $\frac{\gamma}{\beta} \rightarrow +\infty$.\\
We compute the logarithmic derivative of $\frac{\tau}{\beta^2}$ :
\[
\frac{\dot{(\frac{\tau}{\beta^2})}}{\frac{\tau}{\beta^2}} = - \frac{4\alpha}{\beta \gamma - \tau} < 0 
\]
So we get that $\epsilon = \frac{\tau}{\beta^2} \frac{\beta}{\gamma} \rightarrow 0$. Therefore the ratio between the two terms constituting $\dot{\alpha}$ is
\[
\frac{(\beta \gamma - \tau)(\beta^2 - \alpha^2)}{2\alpha^2 \tau} = \frac{(\frac{1}{\epsilon} - 1)(x^2-1)}{2}>1
\]
after a certain time. So $\alpha$ is increasing and $\beta$ is decreasing, with $\beta > \alpha$. They are both converging respectively to $\beta_{\infty} \geq \alpha_{\infty} > 0 $. So $\frac{\gamma}{\beta} \sim \frac{2}{\beta_{\infty}} (x_{\infty} - 2)t$ and $\gamma \sim 2(x_{\infty} - 2)t $. But $\dot{\gamma} \leq -4 + x + \frac{1}{x}$ so we must have $x_{\infty} + \frac{1}{x_{\infty}} - 4 \geq 2x_{\infty} - 4$ so $x_{\infty}\leq 1$, contradiction.
Thus, we've proved the theorem. $\blacksquare$.
\section{Second special case : $SL(2,\mathbb{C})/U(1)$}

\subsection{Lie algebra structure}
The real Lie algebra $g = sl(2,\mathbb{C})$ is generated by the following matrices :
\[
X = \begin{pmatrix}
    i & 0\\
    0 & -i
\end{pmatrix}
\quad
A = \begin{pmatrix}
    1 & 0\\
    0 & -1
\end{pmatrix}
\quad
B = \begin{pmatrix}
    0 & 1\\
    0 & 0
\end{pmatrix}
\]
\[
C = \begin{pmatrix}
    0 & i\\
    0 & 0
\end{pmatrix}
\quad
D = \begin{pmatrix}
    0 & 0\\
    1 & 0
\end{pmatrix}
\quad
E = \begin{pmatrix}
    0 & 0\\
    i & 0
\end{pmatrix}
\]
Of course, $u(1)$ is the subalgebra generated by $X$. The bracket structure is specified in the following table :
\begin{center}
\begin{tabular}{c|cccccc}
$[\quad]$ & $X$ & $A$ & $B$ & $C$ & $D$ & $E$ \\
\hline
$X$ & $0$ & $0$ & $2C$ & $-2B$ & $-2E$ & $2D$ \\
$A$ & $0$ & $0$ & $2B$ & $2C$ & $-2D$ & $-2E$\\
$B$ & $-2C$ & $-2B$ & $0$ & $0$ & $A$ & $X$\\
$C$ & $2B$ & $-2C$ & $0$ & $0$ & $X$ & $-A$ \\
$D$ & $2E$ & $2D$ & $-A$ & $-X$ & $0$ & $0$ \\
$E$ & $-2D$ & $2E$ & $-X$ & $A$ & $0$ & $0$
\end{tabular}
\end{center}
So $p = Vect(A,B,C,D,E)$ is an $Ad(U(1))-$invariant complement of $u(1)$ in $g$. Moreover, it is easy to show the decomposition of $p$ in $\mathfrak{so}(2)$ irreducible representations is $p = Vect(A) + Vect(B,C) + Vect(D,E)$. The problem making this case hard to solve is that the two $2$-dimensional representations are isomorphic, so cross-terms will appear in the metric and the Ricci tensor.\\
It will be useful to have the killing form on $p$ :
\begin{center}
\begin{tabular}{c|ccccc}
B & $A$ & $B$ & $C$ & $D$ & $E$ \\
\hline
$A$ & $16$ & $0$ & $0$ & $0$ & $0$\\
$B$ & $0$ & $0$ & $0$ & $8$ & $0$\\
$C$ & $0$ & $0$ & $0$ & $0$ & $8$ \\
$D$ & $0$ & $8$ & $0$ & $0$ & $0$ \\
$E$ & $0$ & $0$ & $8$ & $0$ & $0$

\end{tabular}
\end{center}
Finally, we see that $tr \circ ad = 0$ so $G$ is unimodular.
\subsection{Homogeneous metrics}
We want to characterize all the the homogeneous metrics on $SL(2,\mathbb{C})/U(1)$. Such a metric correspond to an $\mathrm{Ad}(U(1))-$invariant scalar product on $p$. Using this invariance and the irreducible decomposition of $p$, we get that the scalar product in the basis $(A,B,C,D,E)$ has the following form.
\[
g = \begin{pmatrix}
    \alpha & 0 & 0 & 0 & 0 \\
    0 & \beta & 0 & \mu & \nu \\
    0 & 0 & \beta & -\nu & \mu \\
    0 & \mu & -\nu & \gamma & 0 \\
    0 & \nu & \mu & 0 & \gamma \\
\end{pmatrix}
\]
With $\alpha, \beta, \gamma > 0 $.
As a matter of fact, we can get rid of one parameter.
\begin{lemma}
    Every homogeneous metric on $SL(2,\mathbb{C})/U(1)$ is isometric to a homogeneous metric corresponding to a scalar product on $p$ where $\beta = \gamma$. 
\end{lemma}
\begin{proof}
    Let $x = \mathrm{exp}(tA)$ and $\alpha_t$ the conjugation by $x$ in our Lie group. We know that $(d\alpha_t)_e = \mathrm{Ad}(tA) = \mathrm{\mathrm{exp}}(\mathrm{ad}(tA))$ is the identity on $Vect(E)$ so $\alpha_t (\mathrm{exp}(\lambda X)) = \mathrm{exp}(\lambda X) \forall \lambda$ so $\alpha_t$ is the identity on $U(1)^0 = U(1)$. Therefore, $\alpha_t$ induces $\tilde{\alpha_t}$ on the quotient. Clearly, $\tilde{\alpha_t}$ is a diffeomorphism. Besides, $\tilde{\alpha_t}^*g $ is a $G$-invariant metric on $G/H$. Finally, the corresponding scalar product on $p$ is given by $\tilde{g} = g(\mathrm{Ad}(tA)\cdot, \mathrm{Ad}(tA)\cdot)$.\\
    We compute in the same basis :
    \[
    ad(tA) = \begin{pmatrix}
        0 & 0 & 0 & 0 & 0 \\
        0 & 2t & 0 & 0 & 0 \\
        0 & 0 & 2t & 0 & 0 \\
        0 & 0 & 0 & -2t & 0 \\
        0 & 0 & 0 & 0 & -2t \\
    \end{pmatrix}
    \quad 
    \mathrm{Ad}(\mathrm{exp}(tA)) = \begin{pmatrix}
        1 & 0 & 0 & 0 & 0 \\
        0 & e^{2t} & 0 & 0 & 0 \\
        0 & 0 & e^{2t} & 0 & 0 \\
        0 & 0 & 0 & e^{-2t} & 0 \\
        0 & 0 & 0 & 0 & e^{-2t} \\
    \end{pmatrix}
    \]
    \[
    \tilde{g} = \begin{pmatrix}
        \alpha & 0 & 0 & 0 & 0 \\
        0 & e^{4t}\beta & 0 & \mu & \nu \\
        0 & 0 & e^{4t}\beta & -\nu & \mu \\
        0 & \mu & -\nu & e^{-4t}\gamma & 0 \\
        0 & \nu & \mu & 0 & e^{-4t}\gamma \\
    \end{pmatrix}
    \]
    Then choose $t=\frac{\log \frac{\gamma}{\beta}}{8}$ to finish the proof.
\end{proof}
We now have to find an orthonormal basis for $g$. We introduce $\tau = \mu^2 + \nu^2$. Let's suppose for the moment that $\tau > 0$. We define 
\[
\lambda_{\pm} = \frac{\beta+\gamma}{2} \pm \frac{1}{2} \sqrt{4\tau + (\beta - \gamma)^2}
\]
These are the roots of the equation $(\lambda - \beta)(\lambda - \gamma) = \tau$ and are distinct from $\beta, \gamma$ by assumption.\\
Let's introduce \[
H^{\pm} = \begin{pmatrix}
0\\
    \lambda_{\pm} - \gamma \\
    0\\
    \mu \\
    \nu
\end{pmatrix}
\quad 
B^{\pm} = \begin{pmatrix}
0\\
    0 \\
    \lambda_{\pm} - \gamma\\
    -\nu \\
    \mu
\end{pmatrix}
\]
Then $(H^{\pm}, B^{\pm})$ are eigenvectors of $g$ associated to eigenvalues $\lambda_{\pm}$ (simple computation). It is then straightforward to show that $(\frac{H^{\pm}}{f_{\pm}}, \frac{B^{\pm}}{f_{\pm}})$ is an orthonormal basis for $g$, where \[
f_{\pm} = \sqrt{\lambda_{\pm} \cdot (\tau + (\lambda_{\pm} - \gamma)^2)}
\]
\subsection{Ricci tensor}
We know from [\cite{Besse1987}, Corollary 7.38] that if $(X_i)$ is an orthonormal basis of $p$ for $g$ then the Ricci tensor in the unimodular case is given by :
\[
Ric(X,Y) = -\frac{1}{2}B(X,Y) - \frac{1}{2}\sum_i g([X,X_i],[Y,X_i]) + \frac{1}{4}\sum_{i,j}g([X_i, X_j],X)g([X_i,X_j],Y)
\]
With that formula and simplifications with the $\lambda_{\pm}$, we are able to express the full Ricci tensor (which by the same argument has the same form as the metric $g$):
\[
Ric(A,A) = \frac{\alpha^2 - 16\beta \gamma}{\beta \gamma - \tau}
\]
\[
Ric(B,B) = Ric(C,C) = \frac{\beta (16\tau - \alpha^2)}{2\alpha (\beta \gamma - \tau)}
\]
\[
Ric(D,D) = Ric(E,E) = \frac{\gamma (16\tau - \alpha^2)}{2\alpha (\beta \gamma - \tau)}
\]
\[
Ric(B,D) = Ric(C,E) = -4 + \frac{\mu}{2\alpha (\beta \gamma - \tau)}(-\alpha^2 + 16\beta \gamma)
\]
\[
Ric(B,E) = -Ric(C,D) = \frac{\nu}{2\alpha (\beta \gamma - \tau)}(-\alpha^2 + 16\beta \gamma)
\]
\subsection{Extinction of the Ricci flow}
The Ricci flow equation is $\frac{dg}{dt} = -2Ric(g)$ i.e
\[
\begin{cases}
\displaystyle \dot{\alpha} = 2\frac{16\beta \gamma-\alpha^2}{\beta \gamma - \tau} \\
\displaystyle \dot{\beta} = -\frac{\beta (16\tau - \alpha^2)}{\alpha (\beta \gamma - \tau)} \\
\displaystyle \dot{\gamma} = -\frac{\gamma (16\tau - \alpha^2)}{\alpha (\beta \gamma - \tau)} \\
\displaystyle \dot{\mu} = 8 - \frac{\mu}{\alpha (\beta \gamma - \tau)}(-\alpha^2 + 16\beta \gamma) \\
\displaystyle \dot{\nu} = -\frac{\nu}{\alpha (\beta \gamma - \tau)}(-\alpha^2 + 16\beta \gamma)
\end{cases}
\]
We will now prove the main theorem of this section.
\begin{theorem}
    Any homogeneous Ricci flow on $SL(2,\mathbb{C})/U(1)$ has a finite time extinction.
\end{theorem}
First thing to notice is that if $\beta = \gamma$ at $t=0$ then $\beta = \gamma$ at any time. Combined with the observations in section $2$, we can therefore assume $\beta (t) = \gamma (t)$.\\
Let's suppose that the Ricci flow above is immortal. We should arrive to some contradiction. There are two relevant ratios to analyze, which are $x = \frac{\alpha}{\beta}$ and $\epsilon = \frac{\tau}{\beta \gamma}$.\\
Let's compute their derivatives :
\[
\frac{\dot{x}}{x} = \frac{\beta^2}{\alpha (1-\epsilon)} (16\epsilon + 32 - 3x^2)
\]
\[
\frac{\dot{\epsilon}}{\epsilon} = \frac{16}{\sqrt{\tau}}(\frac{\mu}{\sqrt{\tau}}-\frac{2\sqrt{\epsilon}}{x})
\]
Notice that if $\mu=0$ then $\dot{\mu}=8$. So just before, $\mu <0$ and just after $\mu>0$. Therefore, after a certain time, $\tau>0$ and $\mu \neq 0$.
\begin{claim}
    $\exists t_0$ such that either $\forall t>t_0,x>4$ or $\forall t>t_0,x<4$.
\end{claim}
\begin{proof}
    By the previous expression, if $x=4$ then $\dot{x}<0$. So just before, $x>4$ and just after, $x<4$. Therefore, it can only happen one time and after that, the sign of $x-4$ is preserved.
\end{proof}
\begin{claim}
    There exist $\delta$ such that $\alpha (t)\geq \delta >0$.
\end{claim}
\begin{proof}
    If $x>4$ then $\alpha$ decreases, $\beta$ increases and $\alpha > 4\beta$ so $\alpha, \beta$ converge to $\alpha_{\infty}, \beta_{\infty} >0$.\\
    If $x<4$ then $\alpha$ increases so $\alpha \geq cste >0$.
\end{proof}
\begin{claim}
    We have $\mu \sim \sqrt{\tau}$ and $\frac{\mu}{\sqrt{\tau}}$ is increasing.
\end{claim}
\begin{proof}
    Let's write $\mu(t) = \frac{f(t)}{\sqrt{\alpha(t)}}$ and $\nu (t) = \frac{g(t)}{\sqrt{\alpha(t)}}$. Then 
    \[
    f'(t) = 8\sqrt{\alpha} \geq 8\sqrt{\delta}
    \]
    \[
    g'(t) = 0
    \]
    Thus $\mu(t)\geq \frac{cste + 8\sqrt{\delta}t}{\sqrt{\alpha}}$ and $\nu = \frac{cste}{\sqrt{\alpha}}$ so $\nu(t) = o(\mu(t))$ and $\mu \sim \sqrt{\tau}$. So $\mu>0$ after a certain time and
    \[
    2\frac{\dot{\mu}}{\mu} - \frac{\dot{\tau}}{\tau} = \frac{16}{\mu} - \frac{2(\alpha^2 + 16\beta \gamma)}{\alpha (\beta \gamma - \tau)} - \frac{1}{\tau} (16\mu - \frac{2\tau (\alpha^2 + 16\beta \gamma)}{\alpha (\beta \gamma - \tau)})
    = \frac{16 \tau - 16\mu^2}{\tau \mu} \geq 0
    \]
\end{proof}
\begin{claim}
    x converges to $x_{\infty}>0$.
\end{claim}
\begin{proof}
After a certain time, $\frac{\mu}{\sqrt{\tau}}>\frac{3}{4}$.
    The variations of $x$ are given by :
     \[
\begin{array}{c|ccc}
x & (0, \frac{4}{\sqrt{3}}\sqrt{2+\epsilon}) & \frac{4}{\sqrt{3}}\sqrt{2+\epsilon} & (\frac{4}{\sqrt{3}}\sqrt{2+\epsilon}, \infty)\\
\hline
\dot{(x)} & + & 0 & -
\end{array}
\]
If $x=\frac{4}{\sqrt{3}}\sqrt{2+\epsilon}$, then 
\[
\frac{\dot{\epsilon}}{\epsilon} = \frac{16}{\sqrt{\tau}}(\frac{\mu}{\sqrt{\tau}}-\frac{\sqrt{3}\sqrt{\epsilon}}{2\sqrt{2+\epsilon}}) > \frac{16}{\sqrt{\tau}}(\frac{3}{4}-\frac{\sqrt{3}\sqrt{\epsilon}}{2\sqrt{2+\epsilon}}) > 0
\]
So just before, $x>\frac{4}{\sqrt{3}}\sqrt{2+\epsilon}$ (so $\dot{x}<0$) and just after, $x<\frac{4}{\sqrt{3}}\sqrt{2+\epsilon}$ (so $\dot{x}>0$). It can only happen one time and after that, $x$ is monotonic. Moreover, because of $0\leq \epsilon<1$, if $x>4$ then $x$ decreases and if $x<4\frac{4\sqrt{2}}{\sqrt{3}}$ then $x$ increases. So $x$ converges to a non-zero constant.
\end{proof}
\begin{claim}
    $\epsilon$ converges to a constant $0\leq \epsilon_{\infty} \leq 1$. 
\end{claim}
\begin{proof}
    We rewrite 
    \[
\frac{\dot{\epsilon}}{\epsilon} = \frac{32}{x\sqrt{\tau}}(\frac{x\mu}{2\sqrt{\tau}}-\sqrt{\epsilon})
\]
After a certain time, $\frac{\mu}{\sqrt{\tau}} \geq \frac{\sqrt{3}}{2}$. If $\dot{\epsilon} = 0$ then $\epsilon \geq \frac{3x^2}{16}$ so 
\[
\frac{\dot{x}}{x} = \frac{\beta^2}{\alpha (1-\epsilon^2)} (16\epsilon + 32 - 3x^2) \geq \frac{32\beta^2}{\alpha (1-\epsilon^2)} >0
\]
So right before, $\dot{\epsilon}<0$ and right after $\dot{\epsilon}>0$. Therefore, it can only happen one time and after that, $\epsilon$ is monotonic and bounded so converges.
\end{proof}
If $x_{\infty}>2$ then after a while $\frac{\mu}{\sqrt{\tau}}-\frac{2\sqrt{\epsilon}}{x} \geq \kappa >0$. So 
\[
\frac{\dot{\epsilon}}{\epsilon} \geq \frac{16\kappa}{\sqrt{\tau}} \geq \frac{cste}{cste + 8\sqrt{\delta}t}
\]
By integrating this expression, we find that $\log{\epsilon} \to +\infty$, contradiction.\\
If $x_{\infty}< 2$, then $\epsilon_{\infty}\neq 1$ because if not, $\epsilon$ should increase so $\frac{x\mu}{2\sqrt{\tau}}\geq \sqrt{\epsilon}$ and therefore we would have $x_{\infty} \geq 2$. Thus, \[
\alpha \sim 2\frac{16-x_{\infty}^2}{1-\epsilon_{\infty}}t
\]
\[
\beta \sim \frac{2}{x_{\infty}}\frac{16-x_{\infty}^2}{1-\epsilon_{\infty}}t
\]
\[
\dot{\beta} \to \frac{1}{x_{\infty}}\frac{x_{\infty}^2 - 16\epsilon_{\infty}}{1-\epsilon_{\infty}}
\]
So by equalizing the growth rates :
\[
3x_{\infty}^2 = 32 + 16\epsilon_{\infty}
\]
Finally, $x_{\infty}\geq \frac{4\sqrt{2}}{\sqrt{3}}>2$, contradiction.\\
If $x_{\infty} = 2$ and $\epsilon_{\infty} \neq 1$, then the previous argument also works. The last case is $x_{\infty}=2$ and $\epsilon_{\infty}=1$. By a straightforward argument $\alpha$ increases and $\beta$ decreases, both converge to non zero constants. But then $\frac{\dot{x}}{x} \sim \frac{cste}{1-\epsilon} \to +\infty$ so $x\to +\infty$, contradiction. 
This finishes the proof of the theorem. $\blacksquare$
\section{Contractibility of the space of positively curved metrics}
Here, we give a proof of the \textbf{Corollary 1.3}, stating that the space of all positively curved metrics $M^G_{\mathrm{Scal} >0}$ is either empty or contractible for any five-dimensional homogeneous space $M=G/H$. In fact, we will prove this stronger statement:
\begin{prop}
    Under the dynamical Alekseevski conjecture, $M^G_{\mathrm{Scal} >0}$ is either empty or contractible for any homogeneous space $M=G/H$.
\end{prop}
\begin{proof}
    Let $M$ be a homogeneous space. From the work of Berard-Bergery \cite{berardbergery1978courbure}, if the universal cover of $M$ is diffeomorphic to $\mathbb{R}^n$, then $M^G_{\mathrm{Scal} >0}$ is empty. If not, then by the dynamical Alekseevski conjecture all homogeneous Ricci flows have a finite time extinction. Therefore, using \cite{Lafuente2015}, for every homogeneous metric $g$, the Ricci flow $g(t)$ starting from $g$ reaches a metric with a scalar curvature greater than $1$. Let's define $\mathrm{Scal}(g)$ the scalar curvature of a metric and
    \[
    t_g = \mathrm{inf} \{ s\in \mathbb{R} \mid \mathrm{Scal}(g(s))\geq 1\}.
    \]
    Then,
    \[
    \begin{array}{rcl}
f : & M^G & \longrightarrow M^G_{\mathrm{Scal} \geq 1} \\
    & g & \longmapsto g(t_g)
\end{array}.
    \]
    is continuous because of the continuity of the Ricci flow in its parameters. Moreover, it's a deformation retractation as shown by the following homotopy from $\mathrm{Id}$ to $f$:
     \[
    \begin{array}{rcl}
H : & M^G\times [0,1] & \longrightarrow M^G \\
    & (g,\alpha) & \longmapsto g(\alpha \cdot t_g)
\end{array}.
    \]
    Therefore $M^G_{\mathrm{Scal}\geq 1}$ is homotopy equivalent to $M^G$, which is homotopy equivalent to a point. Besides, the exact same argument shows that $M^G_{\mathrm{Scal}>0}$ and $M^G_{\mathrm{Scal}\geq 1}$ are homotopy equivalent without even using the conjecture. Finally, $M^G_{\mathrm{Scal}>0}$ is contractible.
\end{proof}

\nocite{*}
\bibliographystyle{plain}
\bibliography{refs}
\end{document}